\documentclass[twoside]{article}
\usepackage{enumerate,mathrsfs,amsfonts,amsxtra,latexsym,amssymb,amsthm,verbatim,amsmath,graphicx,dsfont}
\usepackage{mathpazo}
\usepackage[dvipdfm, colorlinks, linkcolor=green, anchorcolor=blue, citecolor=red]{hyperref}

\catcode`@=12

\newtheorem{theorem}{Theorem}[section]
\newtheorem{definition}[theorem]{Definition}
\newtheorem{lemma}[theorem]{Lemma}

\newtheorem{proposition}[theorem]{Proposition}

\begin{document}
\abovedisplayskip=6pt plus 1pt minus 1pt \belowdisplayskip=6pt
plus 1pt minus 1pt
\thispagestyle{empty} \vspace*{-1.0truecm} \noindent
\vskip 10mm

\begin{center}{\large Positive quandle homology and its applications in knot theory} \end{center}

\vskip 5mm
\begin{center}{Zhiyun Cheng \quad Hongzhu Gao\\
{\small School of Mathematical Sciences, Beijing Normal University
\\Laboratory of Mathematics and Complex Systems, Ministry of
Education, Beijing 100875, China
\\(email: czy@bnu.edu.cn \quad hzgao@bnu.edu.cn)}}\end{center}

\vskip 1mm

\noindent{\small {\small\bf Abstract} Algebraic homology and cohomology theories for quandles have been studied extensively in recent years. With a given quandle 2(3)-cocycle one can define a state-sum invariant for knotted curves(surfaces). In this paper we introduce another version of quandle (co)homology theory, say positive quandle (co)homology. Some properties of positive quandle (co)homology groups are given and some applications of positive quandle cohomology in knot theory are discussed.
\ \

\vspace{1mm}\baselineskip 12pt


\noindent{\small\bf MR(2010) Subject Classification} 57M25, 57M27, 57Q45\ \ {\rm }}

\vskip 1mm

\vspace{1mm}\baselineskip 12pt

\section{Introduction}
In knot theory, by considering representations from the knot group onto the dihedral group of order $2n$ one obtain a family of elementary knot invariants, known as Fox $n$-colorings \cite{Fox1961}. Quandle, a set with certain self-distributive operation satisfying axioms analogous to the Reidemeister moves, was first proposed by D. Joyce \cite{Joy1982} and S. V. Matveev \cite{Mat1984} independently. With a given quandle $X$ one can define the quandle coloring invariant by counting the quandle homomorphisms from the fundamental quandle of a knot to $X$. For the fundamental quandle and its presentations the reader is referred to \cite{Joy1982} and \cite{Fen1992}.

Equivalently speaking, one can label each arc of a knot diagram by an element of a fixed quandle, subject to certain constraints. The quandle coloring invariant can be computed by counting ways of these labellings. It is natural to consider how to improve this integral valued knot invariant. Since the quandle coloring invariant equals the number of different proper colorings, it is natural to associate a weight function to each colored knot diagram which does not depend on the choice of the knot diagram. In this way, instead of several colored knot diagrams one will obtain several weight functions and the number of these weight functions is exactly the quandle coloring invariant. In \cite{Car2003} J.S. Carter et al. associate a Boltzmann weight to each crossing and then consider the signed product of Boltzmann weights for all crossing points. In fact based on R. Fenn, C. Rourke and B. Sanderson's framework of rack and quandle homology \cite{Fen1995,Fen1996}, J.S. Carter et al. describe a homology theory for quandles such that each 2-cocycle and 3-cocycle can be used to define a family of invariants of knots and knotted surfaces respectively. Many applications of quandle cocycle invariants have been investigated in the past decade. For example, with a suitable choice of 3-cocycle from the dihedral quandle $R_3$, one can prove the chirality of trefoil \cite{Kam2002}. For knotted surface, by using cocycle invariants it was proved that the 2-twist spun trefoil is non-invertible and has triple point number 4 \cite{Car2003,Sat2004}.

In this paper we introduce another quandle homology and cohomology theory, say positive quandle homology and positive quandle cohomology. The definition of positive quandle (co)homology is similar to that of the original quandle (co)homology. It is not surprising that positive quandle homology shares many common properties with quandle homology, which will be discussed in Section 4. The most interesting part of this new quandle (co)homology theory is that it also can be used to define cocycle invariants for knots and knotted surfaces. Most properties of quandle homology and quandle cocycle invariants have their corresponding versions in positive quandle homology theory. This phenomenon suggests that quandle homology theory and positive quandle homology theory are parallel to each other, and in some special cases (Proposition 3.3) they coincides with each other.

The rest of this paper is arranged as follows: In Section 2, a brief review of quandle structure and quandle coloring invariant is given. Some applications of quandle coloring invariant in knot theory will also be discussed. In Section 3, we give the definition of positive quandle homology and cohomology. The relation between positive quandle (co)homology and quandle (co)homology will also be studied. Section 4 is devoted to the calculation of positive quandle homology and cohomology. We will calculate the positive quandle homology for some simple quandles. In Section 5, we show how to use positive quandle 2-cocycle and 3-cocycle to define invariants for knots and knotted surfaces respectively. We end this paper by two examples which study the trivially colored crossing points of a knot diagram, from where the motivation of this study arises.

\section{Quandle and quandle coloring invariants}
First we take a short review of the definition of quandle.
\begin{definition}
A quandle $(X, \ast)$, is a set $X$ with a binary operation $(a, b)\rightarrow a\ast b$ satisfying the following axioms:
\begin{enumerate}
  \item For any $a\in X$, $a\ast a=a$.
  \item For any $b, c\in X$, there exists a unique $a\in X$ such that $a\ast b=c$.
  \item For any $a, b, c\in X$, $(a\ast b)\ast c=(a\ast c)\ast(b\ast c)$.
\end{enumerate}
\end{definition}

Usually we simply denote a quandle $(X, \ast)$ by $X$. If a non-empty set $X$ with a binary operation $(a, b)\rightarrow a\ast b$ satisfies the second and the third axioms, then we name it a \emph{rack}. In particular if a quandle $X$ satisfies a modified version of the second axiom "for any $b, c\in X$, $(c\ast b)\ast b=c$", i.e. the unique element $a=c\ast b$, we call such quandle an \emph{involutory quandle} \cite{Joy1982}or \emph{kei} \cite{Tak1942}. The relation below follows directly from the definitions above:
\begin{center}
$\{$keis$\}\subset\{$quandles$\}\subset\{$racks$\}$.
\end{center}

In the second axiom we usually denote the element $a$ by $a=c\ast^{-1}b$. It is not difficult to observe that $(X, \ast^{-1})$ also defines a quandle structure, which is usually named as the \emph{dual quandle} of $(X, \ast)$. We denote the dual of $X$ by $X^*$. Note that a quandle is an involutory quandle if and only if $\ast=\ast^{-1}$.

Next we list some most common examples of quandle, see \cite{Fen1992,Ho2005,Joy1982,Ven2012} for more examples.
\begin{itemize}
  \item Trivial quandle of order $n$: $T_n=\{a_1, \cdots, a_n\}$ and $a_i\ast a_j=a_i$.
  \item Dihedral quandle of order $n$: $R_n=\{0, \cdots, n-1\}$ and $i\ast j=2j-i$ $($mod $n)$.
  \item Conjugation quandle: a conjugacy class $X$ of a group $G$ with $a\ast b=b^{-1}ab$.
  \item Alexander quandle: a $Z[t, t^{-1}]$-module $M$ with $a\ast b=ta+(1-t)b$.
\end{itemize}

From now on all the quandles mentioned throughout are assumed to be finite quandles. With a given finite quandle $X$, we can define an associated integer-valued knot invariant Col$_X(K)$, i.e. the quandle coloring invariant. Let $K$ be a knot diagram. We will often abuse our notation, letting $K$ refer both to a knot diagram and the knot itself. It is not difficult to determine the meaning that is intended from the context. A \emph{coloring} of $K$ by a given quandle $X$ is a map from the set of arcs of $K$ to the elements of $X$. We say a coloring is \emph{proper} if at each crossing the images of the map satisfies the relation given in Figure 1.
\begin{center}
\includegraphics{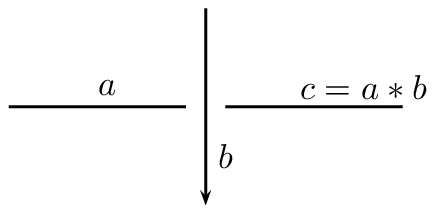}
\centerline{\small Figure 1: The proper coloring rule\quad}
\end{center}

Now we define the \emph{quandle coloring invariant} Col$_X(K)$ to be the number of proper colorings of $K$ by the quandle $X$. Since $X$ is finite, this definition makes sense. It is well-known that although the definition of Col$_X(K)$ depends on the choice of a knot diagram, however the integer Col$_X(K)$ is independent of the knot diagram. In fact the three axioms from the definition of quandle structure correspond to the three Reidemeister moves. In particular Col$_X(K)\geq n$ if $X$ contains $n$ elements, since there always exist $n$ trivial colorings. When $X=R_n$, we have Col$_{R_n}(K)=$Col$_n(K)$, the number of distinct Fox $n$-colorings of $K$ \cite{Fox1961}. It is well-known that Col$_n(K)$ equals the number of distinct representations from the knot group $\pi_1(R^3\backslash K)$ to the dihedral group of order $2n$. As a generalization of Fox $n$-coloring, Col$_X(K)$ is equivalent to the number of quandle homomorphisms from the fundamental quandle of $K$ to $X$. Here the fundamental quandle of $K$ is defined by assigning generators to arcs, and certain relations to crossings, which is quite similar to the presentation of the knot group. See \cite{Joy1982} and \cite{Mat1984} for more details.

Before ending this section we list some properties of the quandle coloring invariant.
\begin{itemize}
  \item Col$_X(K)$=Col$_X(\overline{K^*})$. Here $\overline{K^*}$ denotes the mirror image of $K$ with the reversed orientation. This follows from the fact that the fundamental quandles of $K$ and $\overline{K^*}$ are isomorphic \cite{Joy1982,Mat1984}.
  \item log$_{|X|}($Col$_X(K))\leq b(K)$ and log$_{|X|}($Col$_X(K))\leq u(K)+1$ \cite{Prz1998}. Here $|X|$ denotes the order of $X$, $b(K)$ and $u(K)$ denote the bridge number and unknotting number respectively. The readers are referred to \cite{Cla2013} for some recent progress on the applications of quandle coloring invariants.
  \item Col$_X(K)$ is not a Vassiliev invariant. This can be proved with the similar idea of \cite{Eis1999}, in which M. Eisermann proved that Col$_n(K)$ is not a Vassiliev invariant. Briefly speaking, in \cite{Eis1999} it was proved that if a Vassiliev invariant $F$ is bounded on any given vertical twist sequence, then $F$ is constant. On the other hand, for any fixed vertical twist sequence the braid index is bounded by some integer, say $b$. It is not difficult to show that the fundamental quandle of each knot of this vertical twist sequence can be generated by at most $b$ elements. Assume $X$ contains $n$ elements, then we deduce that Col$_X(K)\leq n^b$. Because Col$_X(K)$ is not constant (note that the choice of the quandle $X$ is arbitrary), therefore Col$_X(K)$ is not a Vassiliev invariant.
\end{itemize}

\section{Homology and cohomology theory for quandles}
Rack (co)homology theory was first defined in \cite{Fen1996}, which is similar to the group (co)homology theory. As a modification of the rack (co)homology, quandle (co)homology was proposed by J.S. Carter, D. Jelsovsky, S. Kamada, L. Langford and M. Saito in \cite{Car2003}. As an application, they defined state-sum invariants for knots and knotted surfaces by using quandle cocycles. Some calculations of quandle homology groups  and the associated state-sum invariants can be found in \cite{Car2001J,Car2001A,Moc2003,Nie2009}, or see \cite{Car2012} for a good survey. First we take a short review of the construction of the quandle (co)homology group, then we will give the definition of positive quandle (co)homology group.

Assume $X$ is a finite quandle. Let $C^R_n(X)$ denote the free abelian group generated by $n-$tuples $(a_1, \cdots, a_n)$, where $a_i\in X$. In order to make $C^R_n(X)$ into a chain complex, let us consider the following two homomorphisms from $C^R_n(X)$ to $C^R_{n-1}(X)$, here $\overline{a_i}$ denotes the omission of the element $a_i$.
\begin{flalign*}
& d_1(a_1, \cdots, a_n)=\sum\limits_{i=1}^n(-1)^i(a_1, \cdots, \overline{a_i}, \cdots, a_n) \quad (n\geq2)\\
& d_2(a_1, \cdots, a_n)=\sum\limits_{i=1}^n(-1)^i(a_1\ast a_i, \cdots, a_{i-1}\ast a_i, a_{i+1}, \cdots, a_n) \quad (n\geq2)\\
& d_i(a_1, \cdots, a_n)=0 \quad (n\leq1, i=1, 2)
\end{flalign*}

For the two homomorphisms $d_1$, $d_2$ defined above, we have the following lemma.
\begin{lemma}
$d_1^2=d_2^2=d_1d_2+d_2d_1=0$.
\end{lemma}
\begin{proof} One computes
\begin{flalign*}
&d_1^2(a_1, \cdots, a_n)&\\
=&d_1(\sum\limits_{i=1}^n(-1)^i(a_1, \cdots, \overline{a_i}, \cdots, a_n))\\
=&\sum\limits_{i=1}^n(-1)^i(\sum\limits_{j<i}(-1)^j(a_1, \cdots, \overline{a_j}, \cdots, \overline{a_i}, \cdots, a_n)+\sum\limits_{j>i}(-1)^{j-1}(a_1, \cdots, \overline{a_i}, \cdots, \overline{a_j}, \cdots, a_n))\\
=&0
\end{flalign*}

\begin{flalign*}
&d_2^2(a_1, \cdots, a_n)&\\
=&d_2(\sum\limits_{i=1}^n(-1)^i(a_1\ast a_i, \cdots, a_{i-1}\ast a_i, a_{i+1}, \cdots, a_n))\\
=&\sum\limits_{i=1}^n(-1)^i(\sum\limits_{j<i}(-1)^j((a_1\ast a_i)\ast (a_j\ast a_i), \cdots, (a_{j-1}\ast a_i)\ast (a_j\ast a_i), {a_{j+1}\ast a_i}, \cdots, a_{i-1}\ast a_i, a_{i+1}, \cdots, a_n)\\
+&\sum\limits_{j>i}(-1)^{j-1}((a_1\ast a_i)\ast a_j, \cdots, (a_{i-1}\ast a_i)\ast a_j, a_{i+1}\ast a_j, \cdots, a_{j-1}\ast a_j, a_{j+1}, \cdots, a_n))\\
=&0
\end{flalign*}

\begin{flalign*}
&d_1d_2(a_1, \cdots, a_n)+d_2d_1(a_1, \cdots, a_n)&\\
=&d_1(\sum\limits_{i=1}^n(-1)^i(a_1\ast a_i, \cdots, a_{i-1}\ast a_i, a_{i+1}, \cdots, a_n))+d_2(\sum\limits_{i=1}^n(-1)^i(a_1, \cdots, \overline{a_i}, \cdots, a_n))\\
=&\sum\limits_{i=1}^n\sum\limits_{j<i}(-1)^{i+j}(a_1\ast a_i, \cdots, \overline{a_j\ast a_i}, \cdots, a_{i-1}\ast a_i, a_{i+1}, \cdots, a_n)\\
+&\sum\limits_{i=1}^n\sum\limits_{j>i}(-1)^{i+j-1}(a_1\ast a_i, \cdots, a_{i-1}\ast a_i, a_{i+1}, \cdots, \overline{a_j}, \cdots, a_n)\\
+&\sum\limits_{i=1}^n\sum\limits_{j<i}(-1)^{i+j}(a_1\ast a_j, \cdots, a_{j-1}\ast a_j, a_{j+1}, \cdots, \overline{a_i}, \cdots, a_n)\\
+&\sum\limits_{i=1}^n\sum\limits_{j>i}(-1)^{i+j-1}(a_1\ast a_j, \cdots, \overline{a_i\ast a_j}, \cdots, a_{j-1}\ast a_j, a_{j+1}, \cdots, a_n)\\
=&0
\end{flalign*}
\end{proof}

Lemma 3.1 suggests us to investigate the following four chain complexes: $\{C^R_n(X), d_1\}$, $\{C^R_n(X), d_2\}$, $\{C^R_n(X), d_1+d_2\}$ and $\{C^R_n(X), d_1-d_2\}$. We remark that $\{C^R_n(X), d_1\}$ is acyclic. In a recent work of A. Inoue and Y. Kabaya \cite{Ino2013}, $\{C^R_n(X), d_1\}$ was regarded as a right $\mathbf{Z}[G_X]-$module, here $G_X$ denotes the associated group of $X$, i.e. $G_X$ is generated by the elements of $X$ and satisfies the relation $a\ast b=b^{-1}ab$. With this viewpoint they defined the simplicial quandle homology to be the homology group of the chain complex $\{C^R_n(X)\bigotimes_{\mathbf{Z}[G_X]}\mathbf{Z}, d_1\}$. The readers are referred to \cite{Ino2013} for more details.

Assume $X$ is a fixed finite quandle. Let $C^D_n(X)$ $(n\geq2)$ denote the free abelian group generated by $n-$tuples $(a_1, \cdots, a_n)$ with $a_i=a_{i+1}$ for some $1\leq i\leq n-1$, and $C^D_n(X)=0$ if $n\leq1$. The following lemma tells us $\{C^D_n(X), d_1\pm d_2\}$ is a sub-complex of $\{C^R_n(X), d_1\pm d_2\}$.
\begin{lemma}
$\{C^D_n(X), d_i\}$ is a sub-complex of $\{C^R_n(X), d_i\}$ $(i=1, 2)$.
\end{lemma}
\begin{proof}
Choose an n-tuple $(a_1, \cdots, a_i, a_{i+1}, \cdots, a_n)\in C^D_n(X)$, where $a_i=a_{i+1}$. One computers
\begin{flalign*}
&d_1(a_1, \cdots, a_i, a_{i+1}, \cdots, a_n)&\\
=&\sum\limits_{j<i}(-1)^j(a_1, \cdots, \overline{a_j}, \cdots, a_i, a_{i+1}, \cdots, a_n)+\sum\limits_{j>i+1}(-1)^j(a_1, \cdots, a_i, a_{i+1}, \cdots, \overline{a_j}, \cdots, a_n)\\
+&(-1)^i(a_1, \cdots, \overline{a_i}, a_{i+1}, \cdots, a_n)+(-1)^{i+1}(a_1, \cdots, a_i, \overline{a_{i+1}}, \cdots, a_n)\\
=&\sum\limits_{j<i}(-1)^j(a_1, \cdots, \overline{a_j}, \cdots, a_i, a_{i+1}, \cdots, a_n)+\sum\limits_{j>i+1}(-1)^j(a_1, \cdots, a_i, a_{i+1}, \cdots, \overline{a_j}, \cdots, a_n)\\
\in &C^D_{n-1}(X)
\end{flalign*}

\begin{flalign*}
&d_2(a_1, \cdots, a_i, a_{i+1}, \cdots, a_n)&\\
=&\sum\limits_{j<i}(-1)^j(a_1\ast a_j, \cdots, a_{j-1}\ast a_j, a_{j+1}, \cdots, a_i, a_{i+1}, \cdots, a_n)\\
+&\sum\limits_{j>i+1}(-1)^j(a_1\ast a_j, \cdots, a_i\ast a_j, a_{i+1}\ast a_j, \cdots, a_{j-1}\ast a_j, a_{j+1}, \cdots, a_n)\\
+&(-1)^i(a_1\ast a_i, \cdots, a_{i-1}\ast a_i, a_{i+1}, \cdots, a_n)+(-1)^{i+1}(a_1\ast a_{i+1}, \cdots, a_i\ast a_{i+1}, a_{i+2}, \cdots, a_n)\\
=&\sum\limits_{j<i}(-1)^j(a_1\ast a_j, \cdots, a_{j-1}\ast a_j, a_{j+1}, \cdots, a_i, a_{i+1}, \cdots, a_n)\\
+&\sum\limits_{j>i+1}(-1)^j(a_1\ast a_j, \cdots, a_i\ast a_j, a_{i+1}\ast a_j, \cdots, a_{j-1}\ast a_j, a_{j+1}, \cdots, a_n)\\
\in &C^D_{n-1}(X)
\end{flalign*}
\end{proof}

Define $C_n^Q(X)=C_n^R(X)/C_n^D(X)$, then we have two chain complexes $\{C_n^Q(X), d_1\pm d_2\}$, here $d_1\pm d_2$ denote the induced homomorphisms. For simplicity, we use $\partial^+$ and $\partial^-$ to denote $d_1+d_2$ and $d_1-d_2$ respectively, and use $C_\ast^{W\pm}(X)$ to denote $\{C_\ast^{W}(X), \partial_\ast^\pm\}$ $(W\in\{R, D, Q\})$. For an abelian group $G$, define the the chain complex $C_\ast^{W\pm}(X; G)$ and cochain complex $C_{W\pm}^{\ast}(X; G)$ as below $(W\in\{R, D, Q\})$
\begin{itemize}
  \item $C_\ast^{W\pm}(X; G)=C_\ast^{W\pm}(X)\bigotimes G$, \quad $\partial_\ast^\pm=\partial_\ast^\pm\bigotimes$ id;
  \item $C_{W\pm}^{\ast}(X; G)=$Hom$(C_\ast^{W\pm}(X), G)$, \quad $\delta_\pm^\ast=$Hom$(\partial_\ast^\pm$, id).
\end{itemize}

The \emph{positive quandle (co)homology groups} of $X$ with coefficient $G$ is defined to be the (co)homology groups of the (co)chain complex $C_\ast^{Q+}(X; G)$ $(C_{Q+}^{\ast}(X; G))$, and the \emph{negative quandle (co)homology groups} of a quandle $X$ with coefficient $G$ is defined to be the (co)homology groups of the (co)chain complex $C_\ast^{Q-}(X; G)$ $(C_{Q-}^{\ast}(X; G))$. In other words,
\begin{center}
$H_n^{Q\pm}(X; G)=H_n(C_\ast^{Q\pm}(X; G))$ and $H^n_{Q\pm}(X; G)=H^n(C_{Q\pm}^{\ast}(X; G))$.
\end{center}
Similarly we can define the \emph{$\pm$ rack (co)homology groups} and \emph{$\pm$ degeneration (co)homology groups} as below,
\begin{center}
$H_n^{R\pm}(X; G)=H_n(C_\ast^{R\pm}(X; G))$ and $H^n_{R\pm}(X; G)=H^n(C_{R\pm}^{\ast}(X; G))$,\\
$H_n^{D\pm}(X; G)=H_n(C_\ast^{D\pm}(X; G))$ and $H^n_{D\pm}(X; G)=H^n(C_{D\pm}^{\ast}(X; G))$.
\end{center}

The reader has recognized that the negative quandle (co)homology groups are nothing but the quandle (co)homology groups introduced by J.S. Carter et al in \cite{Car2003}. Therefore we will still use the name quandle (co)homology instead of nagative quandle (co)homology, and write $H_n^Q(X; G)$ $(H^n_Q(X; G))$instead of $H_n^{Q-}(X; G)$ $(H^n_{Q-}(X; G))$. In the rest of this paper we will focus on the positive quandle homology groups $H_\ast^{Q+}(X; G)$ and cohomology groups $H^\ast_{Q+}(X; G)$. In particular, when $G=\mathbf{Z}_2$, the following result is obvious.
\begin{proposition}
$H_n^{Q+}(X; \mathbf{Z}_2)\cong H_n^Q(X; \mathbf{Z}_2)$ and $H^n_{Q+}(X; \mathbf{Z}_2)\cong H^n_Q(X; \mathbf{Z}_2)$.
\end{proposition}

In the end of this section we list the positive quandle 2-cocycle condition and positive quandle 3-cocycle condition below. Later it will be shown that they are related to the third Reidemeister move of knots and the tetrahedral move of knotted surfaces. The readers are suggested to compare these with the quandle 2-cocycle condition and quandle 3-cocycle condition given in \cite{Car2003}.
\begin{itemize}
  \item A positive quandle 2-cocycle $\phi$ satisfies the condition
  \begin{center}
  $-\phi(b, c)-\phi(b, c)+\phi(a, c)+\phi(a\ast b, c)-\phi(a, b)-\phi(a\ast c, b\ast c)=0$.
  \end{center}
  \item A positive quandle 3-cocycle $\theta$ satisfies the condition
  \begin{center}
  $-\theta(b, c, d)-\theta(b, c, d)+\theta(a, c, d)+\theta(a\ast b, c, d)$\\
  $-\theta(a, b, d)-\theta(a\ast c, b\ast c, d)+\theta(a, b, c)+\theta(a\ast d, b\ast d, c\ast d)=0$.
  \end{center}
\end{itemize}

\section{Computing positive quandle homology and cohomology}
This section is devoted to the calculation of positive quandle homology and cohomology for some simple examples. Before this, we need to discuss some basic properties of the positive quandle homology and cohomology. Most of these results have their corresponding versions in quandle homology and quandle cohomology.

First it was pointed out that since $\{C_n^{Q}(X)\}$ is a chain complex of free abelian groups, there is a universal coefficient theorem for quandle homology and quandle cohomology \cite{Car2001J}. Due to the same reason, there also exists a universal coefficient theorem for positive quandle homology and cohomology.
\begin{theorem}[Universal Coefficient Theorem]
For a given quandle $X$, there are a pair of split exact sequences
\begin{center}
$0\rightarrow H^{Q+}_n(X; \mathbf{Z})\bigotimes G\rightarrow H_n^{Q+}(X; G)\rightarrow \emph{Tor}(H_{n-1}^{Q+}(X; \mathbf{Z}), G)\rightarrow0$,
\end{center}
\begin{center}
$0\rightarrow \emph{Ext}(H_{n-1}^{Q+}(X; \mathbf{Z}), G)\rightarrow H_{Q+}^n(X; G)\rightarrow \emph{Hom}(H_n^{Q+}(X; \mathbf{Z}), G)\rightarrow0$.
\end{center}
\end{theorem}

The universal coefficient theorem tells us that it suffices to study the positive quandle homology and cohomology groups with integer coefficients. As usual we will omit the coefficient group $G$ if $G=\mathbf{Z}$. The following lemma gives an example of the computation of the simplest nontrivial quandle $R_3$ in detail.
\begin{lemma}
$H_{Q+}^2(R_3)\cong \mathbf{Z}_3.$
\end{lemma}
\begin{proof}
Recall that $R_3=\{0, 1, 2\}$ with quandle operations $i\ast j=2j-i$ $($mod 3$)$. Choose a positive quandle 2-cocycle $\phi\in Z^2_{Q+}(R_3)$. We assume that $\phi=\sum\limits_{i, j\in \{0, 1, 2\}}c_{(i, j)}\chi_{(i, j)}$, here $\chi_{(i, j)}$ denotes the characteristic function
\begin{center}
$\chi_{(i, j)}(k, l)=
\begin{cases}
1,& \text{if}\ (i, j)=(k, l);\\
0,& \text{if}\ (i, j)\neq(k, l).
\end{cases}$
\end{center}
Recall that $\phi(i, i)=0$, i.e. $c_{(i, i)}=0$.

Next we need to investigate the positive quandle 2-cocycle conditions
\begin{center}
$-\phi(j, k)-\phi(j, k)+\phi(i, k)+\phi(i\ast j, k)-\phi(i, j)-\phi(i\ast k, j\ast k)=0$
\end{center}
for all triples $(i, j, k)$ from $\{0, 1, 2\}$. There are totally 12 equations on $c_{(i, j)}$.
\begin{center}
$\begin{cases}
-2c_{(1, 0)}+c_{(2, 0)}-c_{(0, 1)}-c_{(0, 2)}=0 \\
-2c_{(2, 0)}+c_{(1, 0)}-c_{(0, 2)}-c_{(0, 1)}=0 \\
-2c_{(0, 1)}+c_{(2, 1)}-c_{(1, 0)}-c_{(1, 2)}=0 \\
-2c_{(2, 1)}+c_{(0, 1)}-c_{(1, 2)}-c_{(1, 0)}=0 \\
-2c_{(0, 2)}+c_{(1, 2)}-c_{(2, 0)}-c_{(2, 1)}=0 \\
-2c_{(1, 2)}+c_{(0, 2)}-c_{(2, 1)}-c_{(2, 0)}=0 \\
-2c_{(1, 2)}+c_{(0, 2)}-c_{(0, 1)}-c_{(1, 0)}=0 \\
-2c_{(2, 1)}+c_{(0, 1)}-c_{(0, 2)}-c_{(2, 0)}=0 \\
-2c_{(0, 2)}+c_{(1, 2)}-c_{(1, 0)}-c_{(0, 1)}=0 \\
-2c_{(2, 0)}+c_{(1, 0)}-c_{(1, 2)}-c_{(2, 1)}=0 \\
-2c_{(0, 1)}+c_{(2, 1)}-c_{(2, 0)}-c_{(0, 2)}=0 \\
-2c_{(1, 0)}+c_{(2, 0)}-c_{(2, 1)}-c_{(1, 2)}=0
\end{cases}$
\end{center}
After simplifying the equations above we obtain
\begin{center}
$\begin{cases}
c_{(0, 1)}=z\\
c_{(1, 0)}=-y-z\\
c_{(0, 2)}=y\\
c_{(2, 0)}=-y-z\\
c_{(1, 2)}=y\\
c_{(2, 1)}=z
\end{cases}$
\end{center}
Here we put $c_{(1, 2)}=y$ and $c_{(2, 1)}=z$. Hence the positive quandle 2-cocycle
\begin{center}
$\phi=y(\chi_{(0, 2)}+\chi_{(1, 2)}-\chi_{(1, 0)}-\chi_{(2, 0)})+z(\chi_{(0, 1)}+\chi_{(2, 1)}-\chi_{(1, 0)}-\chi_{(2, 0)})$.
\end{center}
On the other hand, we have
\begin{center}
$\begin{cases}
\delta\chi_0=(\chi_{(0, 2)}+\chi_{(1, 2)}-\chi_{(1, 0)}-\chi_{(2, 0)})+(\chi_{(0, 1)}+\chi_{(2, 1)}-\chi_{(1, 0)}-\chi_{(2, 0)})\\
\delta\chi_1=(\chi_{(1, 0)}+\chi_{(2, 0)}-\chi_{(0, 1)}-\chi_{(2, 1)})+(\chi_{(0, 2)}+\chi_{(1, 2)}-\chi_{(0, 1)}-\chi_{(2, 1)})\\
\delta\chi_2=(\chi_{(0, 1)}+\chi_{(2, 1)}-\chi_{(0, 2)}-\chi_{(1, 2)})+(\chi_{(1, 0)}+\chi_{(2, 0)}-\chi_{(0, 2)}-\chi_{(1, 2)}).
\end{cases}$
\end{center}

Since
\begin{center}
$\phi=y(\delta\chi_0)+(z-y)(\chi_{(0, 1)}+\chi_{(2, 1)}-\chi_{(1, 0)}-\chi_{(2, 0)})$,
\end{center}
then
\begin{center}
$H_{Q+}^2(R_3)\cong\{\chi_{(0, 1)}+\chi_{(2, 1)}-\chi_{(1, 0)}-\chi_{(2, 0)}\mid \delta\chi_0, \delta\chi_1\}$
\end{center}
From $\delta\chi_0=\delta\chi_1=0$ one can easily deduce that $3(\chi_{(0, 1)}+\chi_{(2, 1)}-\chi_{(1, 0)}-\chi_{(2, 0)})=0$. It follows that $H_{Q+}^2(R_3)\cong \mathbf{Z}_3.$
\end{proof}

We remark that the second quandle cohomology group of $R_3$ is trivial, $H_Q^2(R_3; \mathbf{Z})\cong 0$ \cite{Car2003}.

According to the definition $C_n^Q(X)=C_n^R(X)/C_n^D(X)$, there is a short exact sequence
\begin{center}
$0\rightarrow C_\ast^D(X)\rightarrow C_\ast^R(X)\rightarrow C_\ast^Q(X)\rightarrow0$
\end{center}
of chain complexes, it follows that there is a long exact sequence of homology groups
\begin{center}
$\cdots\rightarrow H_n^D(X)\rightarrow H_n^R(X)\rightarrow H_n^Q(X)\rightarrow H_{n-1}^D(X)\rightarrow\cdots$
\end{center}
In \cite{Car2001J}, it was conjectured that the short exact sequence of chain complexes above is split. Later R.A. Litherland and S. Nelson gave an affirmative answer to this conjecture in \cite{Lit2003}. The following theorem says that the splitting map defined by R.A. Litherland and S. Nelson still works in positive quandle homology theory.
\begin{theorem}
For a given quandle $X$, there exists a short exact sequence
\begin{center}
$0\rightarrow H_n^{D+}(X)\rightarrow H_n^{R+}(X)\rightarrow H_n^{Q+}(X)\rightarrow0$.
\end{center}
\end{theorem}
\begin{proof}
According to the definition of positive homology groups there exists a short exact sequence
\begin{center}
$0\rightarrow C_\ast^{D+}(X)\xrightarrow{u_\ast} C_\ast^{R+}(X)\xrightarrow{v_\ast} C_\ast^{Q+}(X)\rightarrow0$.
\end{center}
It suffices to find a chain map $w_n: C_n^{R+}(X)\rightarrow C_n^{D+}(X)$ such that $w_n\circ u_n=id$. Here we use the splitting map $w_n(c)=c-\alpha_n(c)$ introduced by R.A. Litherland and S. Nelson in \cite{Lit2003}, $c\in C_n^{R+}(X)$, and $\alpha_n$ is defined by $\alpha_n(a_1, \cdots, a_n)=(a_1, a_2-a_1, \cdots, a_n-a_{n-1})$ on $n-$tuples and extending linearly to $C_n^{R+}(X)$. The following two relationships will be frequently used during the proof, which also can be found in \cite{Lit2003}. Note that the notation we use here is a bit different from that in \cite{Lit2003}.
\begin{itemize}
\item $\partial^+(a_1, \cdots, a_{n+1})=(\partial^+(a_1, \cdots, a_n), a_{n+1})+(-1)^{n+1}((a_1, \cdots, a_n)+(a_1, \cdots, a_n)\ast a_{n+1})$, here the notation $(a_1, \cdots, a_n)\ast a_{n+1}$ denotes $(a_1\ast a_{n+1}, \cdots, a_{n}\ast a_{n+1})$.
\item $\alpha_{n+1}(a_1, \cdots, a_{n+1})=(\alpha_{n}(a_1, \cdots, a_n), a_{n+1})-(\alpha_n(a_1, \cdots, a_n), a_n)$. Generally, we write
\begin{center}
$\alpha_{n+1}(c, a_{n+1})=(\alpha_n(c), a_{n+1})-(\alpha_n(c), l(c))$,
\end{center}
here $c\in C_n^{R+}(X)$ and $l(c)\in C_1^{R+}(X)$. In particular $l(a_1, \cdots, a_n)=a_n$.
\end{itemize}

First we show that $c-\alpha_n(c)\in C_n^{D+}(X)$ and $w_n\circ u_n=id$. In order to prove $c-\alpha_n(c)\in C_n^{D+}(X)$ it is sufficient to consider the case $c=(a_1, \cdots, a_n)\in C_n^{R+}(X)$. Note that $a_1-\alpha_1(a_1)=a_1-a_1=0\in C_1^{D+}(X)$ and $(a_1, a_2)-\alpha_2(a_1, a_2)=(a_1, a_2)-(a_1, a_2) +(a_1, a_1)=(a_1, a_1)\in C_2^{D+}(X)$. Suppose $c-\alpha_n(c)\in C_n^{D+}(X)$ for some $n$, consider
\begin{flalign*}
&(a_1, \cdots, a_{n+1})-\alpha_{n+1}(a_1, \cdots, a_{n+1})&\\
=&(a_1, \cdots, a_{n+1})-(\alpha_n(a_1, \cdots, a_n), a_{n+1})+(\alpha_n(a_1, \cdots, a_n), a_n)\\
=&(a_1, \cdots, a_{n+1})-(\alpha_n(a_1, \cdots, a_n), a_{n+1})-(a_1, \cdots, a_n, a_n)+(\alpha_n(a_1, \cdots, a_n), a_n)+(a_1, \cdots, a_n, a_n)\\
=&((a_1, \cdots, a_n)-\alpha_n(a_1, \cdots, a_n), a_{n+1})-((a_1, \cdots, a_n)-\alpha_n(a_1, \cdots, a_n), a_n)+(a_1, \cdots, a_n, a_n)\\
\in& C_n^{D+}(X).
\end{flalign*}
In order to show that $w_n\circ u_n=id$, choose $c=(a_1, \cdots, a_i, a_{i+1}, \cdots, a_n)\in C_n^{D+}(X)$, where $a_i=a_{i+1}$, it suffices to prove that $\alpha_n(c)=0$. In fact
\begin{center}
$\alpha_n(c)=(a_1, a_2-a_1, \cdots, a_{i+1}-a_i, \cdots, a_n-a_{n-1})=0$.
\end{center}

Next we shoe that $w_n: C_n^{R+}(X)\rightarrow C_n^{D+}(X)$ is a chain map. We need the two equalities below $(n\geq2)$:
\begin{flalign*}
\alpha_n(d_1(a_1, \cdots, a_n), a_n)=&\alpha_n(-(a_2, \cdots, a_n, a_n)+\cdots+(-1)^n(a_1, \cdots, a_n))&\\
=&(-1)^n\alpha_n(a_1, \cdots, a_n)
\end{flalign*}
\begin{flalign*}
\alpha_n(d_2(a_1, \cdots, a_n), a_n)=&\alpha_n(\sum\limits_{i=1}^n(-1)^i(a_1\ast a_i, \cdots, a_{i-1}\ast a_i, a_{i+1}, \cdots, a_n, a_n))&\\
=&(-1)^n\alpha_n((a_1, \cdots, a_n)\ast a_n)
\end{flalign*}

Now we show that $\partial_{n+1}^+\alpha_{n+1}-\alpha_n\partial_{n+1}^+=0$. First note that
\begin{center}
$\partial_2^+\alpha_2(a_1, a_2)=-(a_2)-(a_2)+(a_1)+(a_1\ast a_2)=\alpha_1\partial_2^+(a_1, a_2)$.
\end{center}
Assume $\partial_{n+1}^+\alpha_{n+1}-\alpha_n\partial_{n+1}^+=0$ holds for some $n\geq 2$, one computes
\begin{flalign*}
&\partial_{n+1}^+\alpha_{n+1}(a_1, \cdots, a_{n+1})-\alpha_n\partial_{n+1}^+(a_1, \cdots, a_{n+1})&\\
=&\partial_{n+1}^+((\alpha_n(a_1, \cdots, a_n), a_{n+1})-(\alpha_n(a_1, \cdots, a_n), a_n))\\
&-\alpha_n((\partial_n^+(a_1, \cdots, a_n), a_{n+1})+(-1)^{n+1}(a_1, \cdots, a_n)+(-1)^{n+1}(a_1, \cdots, a_n)\ast a_{n+1})\\
=&(\partial_n^+\alpha_n(a_1, \cdots, a_n), a_{n+1})+(-1)^{n+1}\alpha_n(a_1, \cdots, a_n)+(-1)^{n+1}\alpha_n(a_1, \cdots, a_n)\ast a_{n+1}\\
&-(\partial_n^+\alpha_n(a_1, \cdots, a_n), a_n)-(-1)^{n+1}\alpha_n(a_1, \cdots, a_n)-(-1)^{n+1}\alpha_n(a_1, \cdots, a_n)\ast a_n\\
&-(\alpha_{n-1}\partial_n^+(a_1, \cdots, a_n), a_{n+1})-(\alpha_{n-1}\partial_n^+(a_1, \cdots, a_n), l(\partial_n^+(a_1, \cdots, a_n)))\\
&-(-1)^{n+1}\alpha_n(a_1, \cdots, a_n)-(-1)^{n+1}\alpha_n((a_1, \cdots, a_n)\ast a_{n+1})\\
=&-(\alpha_{n-1}\partial_n^+(a_1, \cdots, a_n), a_n)-(-1)^{n+1}\alpha_n(a_1, \cdots, a_n)\\
&-(-1)^{n+1}\alpha_n(a_1, \cdots, a_n)\ast a_n-(\alpha_{n-1}\partial_n^+(a_1, \cdots, a_n), l(\partial_n^+(a_1, \cdots, a_n)))\\
=&-\alpha_n(\partial_n^+(a_1, \cdots, a_n), a_n)-(-1)^{n+1}\alpha_n(a_1, \cdots, a_n)-(-1)^{n+1}\alpha_n(a_1, \cdots, a_n)\ast a_n\\
=&-\alpha_n((d_1+d_2)(a_1, \cdots, a_n), a_n)-(-1)^{n+1}\alpha_n(a_1, \cdots, a_n)-(-1)^{n+1}\alpha_n(a_1, \cdots, a_n)\ast a_n\\
=&-(-1)^n\alpha_n(a_1, \cdots, a_n)-(-1)^n\alpha_n(a_1, \cdots, a_n)\ast a_n\\
&-(-1)^{n+1}\alpha_n(a_1, \cdots, a_n)-(-1)^{n+1}\alpha_n(a_1, \cdots, a_n)\ast a_n\\
=&0
\end{flalign*}
\end{proof}

Now we investigate $H_1^{Q+}(X)$ and $H_2^{Q+}(X)$ for general quandle $X$. The similar results of quandle homology groups can be found in \cite{Car2001J} and \cite{Kam2002}. Assume $X=\{a_1, \cdots, a_n\}$, according to the definitions of $d_1$ and $d_2$ we have $Z_1^{Q+}(X)=C_1^{Q+}(X)=C_1^{R+}(X)$, i.e. the free abelian group generated by $\{a_1, \cdots, a_n\}$. Since $\partial_2^+(a, b)=-b-b+a+a\ast b$, we conclude that
\begin{center}
$H_1^{Q+}(X)\cong\{a_1, \cdots, a_n\mid a_i\ast a_j=2a_j-a_i\}$.
\end{center}
\begin{proposition}
$H_1^{Q+}(T_n)\cong\mathbf{Z}\bigoplus(\bigoplus\limits_{n-1}\mathbf{Z}_2)$ and $H_1^{Q+}(R_n)\cong\mathbf{Z}\bigoplus \mathbf{Z}_n$.
\end{proposition}
\begin{proof}
According to the analysis above, we have
\begin{center}
$H_1^{Q+}(T_n)\cong\{a_1, \cdots, a_n\mid 2a_i=2a_j\}\cong\{a_1, a_2-a_1, \cdots, a_n-a_1\mid 2(a_i-a_1)=0\}\cong\mathbf{Z}\bigoplus(\bigoplus\limits_{n-1}\mathbf{Z}_2)$.
\end{center}

For the dihedral quandle $R_n=\{a_0, \cdots, a_{n-1}\}$ with quandle operations $a_i\ast a_j=a_{2j-i\ (mod\ n)}$, we have
\begin{center}
$H_1^{Q+}(R_n)\cong\{a_0, \cdots, a_{n-1}\mid a_{2j-i\ (mod\ n)}=2a_j-a_i\}\cong\{a_0, a_1-a_0\mid n(a_1-a_0)=0\}\cong\mathbf{Z}\bigoplus \mathbf{Z}_n$.
\end{center}
\end{proof}

Next we study the second positive degeneration homology $H_2^{D+}(X)$. Given a quandle $X$ and $\{a, b\}\in X$, we define $a\sim b$ if there exists some elements $a_1, \cdots, a_n$ of $X$ such that $b=(\cdots ((a\ast^{\varepsilon_1} a_1)\ast^{\varepsilon_2}a_2)\cdots )\ast^{\varepsilon_n} a_n$, where $\varepsilon_i\in \{\pm1\}$. The \emph{orbits} of $X$ are defined to be the set of equivalence classes of $X$ by $\sim$. We denote it by Orb$(X)$, and as usual the number of elements in Orb$(X)$ is denoted by $|$Orb$(X)|$. Since $\partial^+(a, a)=-a-a+a+a=0$, and
\begin{flalign*}
&\partial^+(a, a, b)=-2(a, b)+(a, b)+(a, b)-(a, a)-(a\ast b, a\ast b)=-(a, a)-(a\ast b, a\ast b),\\
&\partial^+(a, b, b)=-2(b, b)+(a, b)+(a\ast b, b)-(a, b)-(a\ast b, b)=-2(b, b).
\end{flalign*}
Combining with Theorem 4.3, it follows that
\begin{proposition}
$H_2^{D+}(X)\cong\bigoplus\limits_{|Orb(X)|}\mathbf{Z}_2$ and $H_2^{R+}(X)\cong H_2^{Q+}(X)\bigoplus(\bigoplus\limits_{|Orb(X)|}\mathbf{Z}_2)$.
\end{proposition}

In the end of this section let us turn to the trivial quandle $T_n$. In quandle homology theory, the boundary operators of $T_n$ are trivial, therefore $H_n^Q(T_n)\cong C_n^Q(T_n)$. However in positive quandle homology theory, the boundary operators are not trivial in general. In fact we have the following proposition.
\begin{proposition}
$H_i^{Q+}(T_n)\cong
\begin{cases}
\mathbf{Z}\bigoplus(\bigoplus\limits_{n-1}\mathbf{Z}_2),& \ i=1;\\
\bigoplus\limits_{(n-1)^i}\mathbf{Z}_2,& \ i\geq2,
\end{cases}$
and
$H^i_{Q+}(T_n)\cong
\begin{cases}
\mathbf{Z},& \ i=1;\\
\bigoplus\limits_{(n-1)^{i-1}}\mathbf{Z}_2,& \ i\geq2.
\end{cases}$
\end{proposition}
\begin{proof}
It suffices to compute $H_i^{Q+}(T_n)$, $H_{Q+}^i(T_n)$ can be deduced from the universal coefficient theorem. For the case $i=1$, the result follows from Proposition 4.4.

Now we show that $H_2^{Q+}(T_n)\cong\bigoplus\limits_{(n-1)^2}\mathbf{Z}_2$, recall that $T_n=\{a_1, \cdots, a_n\}$ with quandle operations $a_i\ast a_j=a_i$. Notice that $\partial_2^+(a_i, a_j)=-2a_j+a_i+a_i\ast a_j=2(a_i-a_j)$, therefore any element $\psi\in Z_2^{Q+}(T_n)$ can be wrote as $\psi=\sum\limits_{i=1}^nc_i\psi_i$, where $\psi_i=(a_{i_1}, a_{i_2})+\cdots+(a_{i_{k-1}}, a_{i_k})+(a_{i_k}, a_{i_1})$. It follows that $Z_2^{Q+}(T_n)$ can be generated by
\begin{center}
$\{(a_i, a_j)+(a_j, a_i), (a_1, a_i)+(a_i, a_j)+(a_j, a_1)\}$ \quad $(1\leq i<j\leq n)$,
\end{center}
which is equivalent to
\begin{center}
$\{(a_1, a_i)+(a_i, a_j)+(a_j, a_1)\}$ \quad $(2\leq i\leq j\leq n)$.
\end{center}
On the other hand, since
\begin{center}
$\partial^+(a_i, a_j, a_k)=2(-(a_j, a_k)+(a_i, a_k)-(a_i, a_j))$ and $\partial^+(a_i, a_j, a_i)=2(-(a_j, a_i)-(a_i, a_j)),$
\end{center}
we have
\begin{flalign*}
H_2^{Q+}(T_n)\cong&\{(a_1, a_i)+(a_i, a_j)+(a_j, a_1)\mid 2((a_i, a_j)+(a_j, a_i)), 2((a_i, a_j)+(a_j, a_k)-(a_i, a_k))\}&\\
\cong&\{(a_1, a_i)+(a_i, a_j)+(a_j, a_1)\mid 2((a_1, a_i)+(a_i, a_j)+(a_j, a_1))\}\\
\cong&\bigoplus\limits_{(n-1)^2}\mathbf{Z}_2
\end{flalign*}

Similarly since $\partial_i^+=2d_1$ for $C_i^Q(T_n)$, it is not difficult to observe that $($here $2\leq j_k\leq n)$
\begin{flalign*}
H_i^{Q+}(T_n)\cong&\{\frac{1}{2}(\partial_{i+1}^+(a_1, a_{j_1}, \cdots, a_{j_i}))\mid \partial_{i+1}^+(a_1, a_{j_1}, \cdots, a_{j_i})\}&\\
\cong&\bigoplus\limits_{(n-1)^i}\mathbf{Z}_2
\end{flalign*}
\end{proof}

\section{Knot invariants derived from positive quandle cocycles}
\subsection{Positive quandle cocycle invariants for knots}
One of the most important applications of quandle cohomology groups is that one can define knot invariants via quandle 2-cocycles and knotted surface invariants via quandle 3-cocycles. In this section we will show that positive quandle 2-cocycles can also be used to define knot invariants, which is similar to the definition of quandle cocycle invariants introduced in \cite{Car2003}.

Let $K$ be a oriented knot diagram and $X$ a finite quandle. Assume $G$ is an abelian group and $\phi\in H_{Q+}^2(X; G)$ is a positive quandle 2-cocycle. It is well-known that all regions of $R^2-K$ can be colored with white and black in checkerboard fashion such that the unbounded region gets the white color. For each crossing point $\tau$ we can associate a sign $\epsilon(\tau)$ as the figure below.
\begin{center}
\includegraphics{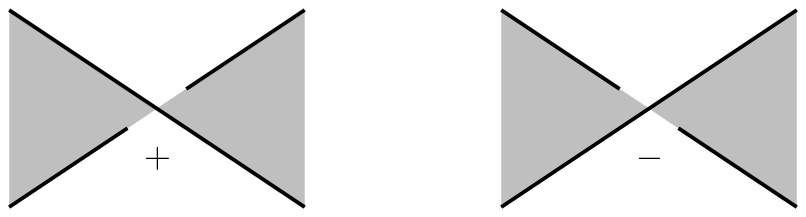}
\centerline{\small Figure 2: The signs of crossings\quad}
\end{center}
Let $\rho$ be a proper coloring of $K$ by $X$, i.e. a homomorphism from the fundamental quandle of $K$ to $X$. In other words, each arc of the diagram is labelled with an element of $X$. For each crossing point $\tau$, assume the over-arc and under-arcs at $\tau$ are colored by $b$ and $a, a\ast b$ respectively, see Figure 1. We consider a weight which is an element of $G$ as
\begin{center}
$W_{\phi}(\tau, \rho)=\phi(a, b)^{\epsilon(\tau)}$,
\end{center}
where $\epsilon(\tau)=\pm1$ according to Figure 2. Then we define the \emph{positive quandle 2-cocycle invariant} of $K$ to be
\begin{center}
$\Phi_{\phi}(K)=\sum\limits_{\rho}\prod\limits_{\tau}W_{\phi}(\tau, \rho)\in \mathbf{Z}G$,
\end{center}
where $\rho$ runs all proper colorings of $K$ by $X$ and $\tau$ runs all crossing points of the diagram. Note that if the sign of the crossing $\epsilon(\tau)$ is replaced by the writhe of $\tau$, one obtains the state-sum $($associated with a quandle 2-cocycle $\phi)$ knot invariants defined by J.S. Carter et al. in \cite{Car2003}.
\begin{theorem}
The positive quandle 2-cocycle invariant $\Phi_{\phi}(K)$ is preserved under Reidemeister moves. If a pair of positive quandle 2-cocycles $\phi_1$ and $\phi_2$ are cohomologous, then $\Phi_{\phi_1}(K)=\Phi_{\phi_2}(K)$. In particular if $\phi$ is a coboundary, we have $\Phi_{\phi}(K)=\sum\limits_{Col_X(K)}1$.
\end{theorem}
\begin{proof}
First we prove that $\Phi_{\phi}(K)$ is invariant under Reidemeister moves. In \cite{Pol2010}, M. Polyak proved that all the classical Reidemeister moves can be realized by a generating set of four Reidemeister moves: $\{\Omega_{1a}, \Omega_{1b}, \Omega_{2a}, \Omega_{3a}\}$, see Figure 3. Hence it suffices to show that $\Phi_{\phi}(K)$ is invariant under $\Omega_{1a}, \Omega_{1b}, \Omega_{2a}$ and $\Omega_{3a}$.
\begin{center}
\includegraphics{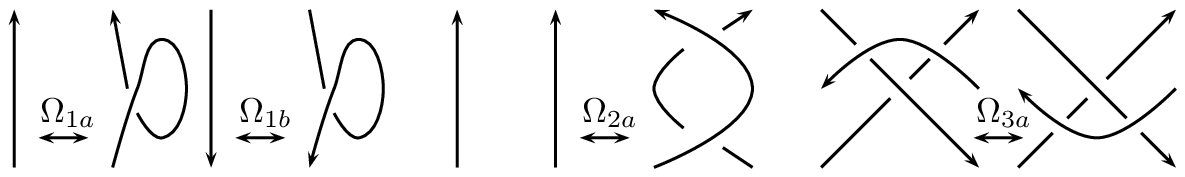}
\centerline{\small Figure 3: Reidemeister moves\quad}
\end{center}
\begin{itemize}
\item $\Omega_{1a}$ and $\Omega_{1b}$: the weight assigned to the crossing point in $\Omega_{1a}$ or $\Omega_{1b}$ is of the form $\phi(a, a)^{\pm1}$, according to the definition of positive quandle cocycle we have $\phi(a, a)^{\pm1}=1$.
\item $\Omega_{2a}$: assume the two arcs on the left side are colored by $a, b$ respectively, then the sum of the weights of the two crossing points on the right side is $\phi(b, a)\phi(b, a)^{-1}=1$.
\item $\Omega_{3a}$: without loss of generality, we assume the top region on both sides are colored white. Under this assumption the signs of each crossings are shown in the figure below.
\begin{center}
\includegraphics{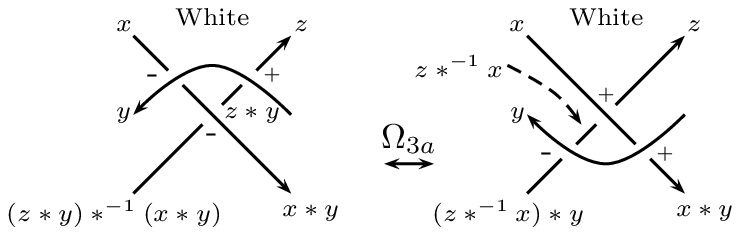}
\centerline{\small Figure 4: Proper colorings under $\Omega_{3a}$\quad}
\end{center}
In order to show that $\Phi_{\phi}(K)$ is invariant under $\Omega_{3a}$, it is sufficient to prove that
\begin{center}
$\phi(x, y)^{-1}\phi(z, y)\phi((z\ast y)\ast^{-1}(x\ast y), x\ast y)^{-1}=\phi(z\ast^{-1}x, y)^{-1}\phi(z\ast^{-1}x, x)\phi(x, y)$.
\end{center}
Note that $(z\ast y)\ast^{-1}(x\ast y)=(z\ast^{-1}x)\ast y$. Put $(a, b, c)=(z\ast^{-1}x, x, y)$ and compare the equation with the positive quandle 2-cocycle condition (note that the equation is written in multiplicative notation here), the result follows.
\end{itemize}

In order to finish the proof it suffices to show that $\Phi_{\phi}(K)=\sum\limits_{Col_X(K)}1$ if $\phi$ is a coboundary. Assume $\phi=\delta_+^1\varphi$ for some $\varphi\in C_{Q+}^1(X; G)$, then
\begin{center}
$\phi(a, b)=\delta_+^1\varphi(a, b)=\varphi(\partial_2^{+}(a, b))=\varphi(-2(b)+(a)+(a\ast b))=\varphi(b)^{-2}\varphi(a)\varphi(a\ast b)\in G$.
\end{center}
First let us consider the simplest case, we assume the knot diagram is alternating, therefore all crossings have the same sign. Without loss of generality all the crossings are assumed to be positive. In this case for a given arc $\lambda$ of the knot diagram, there exists only one crossing such that $\lambda$ is the over-arc at this crossing. On the other hand, this arc is the under-arc at two crossings. For a fixed proper coloring $\rho$, suppose the labelled element of $\lambda$ is $a\in X$, then the contribution of $\lambda$ to $\prod\limits_\tau W_\phi(\tau, \rho)$ comes from the three crossing points that $\lambda$ involved, which equals $\varphi(a)^{-2}\varphi(a)\varphi(a)=1$. It follows that $\prod\limits_\tau W_\phi(\tau, \rho)=1$, hence $\Phi_{\phi}(K)=\sum\limits_{Col_X(K)}1$. The proof of the non-alternating case is analogous to the alternating case. In fact it suffices to notice that if an arc $\lambda$ is the over-arc at several crossings, then the signs of these crossings are alternating. It is not difficult to find that the contribution of $\lambda$ to $\prod\limits_\tau W_\phi(\tau, \rho)$ is still trivial. The proof is finished.
\end{proof}

Recall that in quandle cohomology theory $H_Q^2(R_3)=0$, it means quandle 2-cocycle invariant of $R_3$ can not offer any more information than the Fox 3-colorings. In fact it was pointed out in \cite{Car2003} that all knots have trivial quandle 2-cocycle invariants with any dihedral quandle $R_n$ and any quandle 2-cocycle. We remark that although quandle 2-cocycle invariants of $R_n$ are trivial, some quandle 3-cocycle of $H_Q^3(R_3; \mathbf{Z}_3)$ can be used to distinguish trefoil and its mirror image \cite{Rou2000}.
\begin{proposition}
All knots have trivial positive quandle 2-cocycle invariants with any dihedral quandle $R_n$, associated with any positive quandle 2-cocycle $\phi\in H_{Q+}^2(R_n)$.
\end{proposition}
\begin{proof}
If $n$ is even, according to the coloring rule at each crossing point, for each colored knot diagram all the assigned elements have the same parity. If all assigned elements are even, then by replacing the assigned element $i$ with $\frac{i}{2}$ we obtain a proper coloring with $R_{\frac{n}{2}}$. Consider the element $\phi'$ of $H_{Q+}^2(R_{\frac{n}{2}})$ defined by $\phi'(i, j)=\phi(2i, 2j)$, then $\Phi_{\phi'}(K)$ with $R_{\frac{n}{2}}$ is nontrivial if $\Phi_{\phi}(K)$ with $R_n$ is nontrivial. If all assigned elements are odd, then one obtains a proper coloring with $R_{\frac{n}{2}}$ by replacing each labelled element $i$ with $\frac{i-1}{2}$. Similarly if $\Phi_{\phi}(K)$ with $R_n$ is nontrivial then $\Phi_{\phi''}(K)$ with $R_{\frac{n}{2}}$ is also nontrivial, where $\phi''(i, j)=\phi(2i+1, 2j+1)$. Therefore it is sufficient to consider the case of odd $n$.

If $n$ is odd, it suffices to prove that the free part of $H_{Q+}^2(R_n)=0$. This follows from a general fact: $\Phi_{\phi}(K)$ is trivial if $\phi$ has finite order in $H_{Q+}^2(X)$. In fact assume $k\phi=0\in H_{Q+}^2(X)$, then $\prod\limits_\tau W_{k\phi}(\tau, \rho)=0$. In other words, $\prod\limits_\tau k\phi(a, b)^{\epsilon(\tau)}=k(\prod\limits_\tau\phi(a, b)^{\epsilon(\tau)})=0$. Since we are working with the coefficient $\mathbf{Z}$, it follows that $\prod\limits_\tau\phi(a, b)^{\epsilon(\tau)}=0$.

Assume the free part of $H_{Q+}^2(R_n)\neq0$, it follows that the free part of $H^{Q+}_2(R_n)\neq0$. Replacing the coefficient $\mathbf{Z}$ by $\mathbf{Z}_2$ one concludes that $H^{Q+}_2(R_n; \mathbf{Z}_2)$ contains $\mathbf{Z}_2$ as a summand. By Proposition 3.3 we have $H_2^Q(R_n; \mathbf{Z}_2)=\mathbf{Z}_2\bigoplus$ else. However since $H_2^Q(R_n; \mathbf{Z})=0$ \cite{Car2001J} and $H_1^Q(R_n; \mathbf{Z})=\mathbf{Z}$, the universal coefficient theorem tells us that $H_2^Q(R_n; \mathbf{Z}_2)=0$. The proof is finished.
\end{proof}

Now we give a non-trivial example of positive quandle 2-cocycle invariant. With the matrix of a finite quandle introduced in \cite{Ho2005}, quandle $S_4$ contains four elements $\{0, 1, 2, 3\}$ with quandle operations
\begin{center}
$\begin{bmatrix}
  0 & 2 & 3 & 1 \\
  3 & 1 & 0 & 2 \\
  1 & 3 & 2 & 0 \\
  2 & 0 & 1 & 3 \\
\end{bmatrix}$,
\end{center}
where the entry in row $i$ column $j$ denotes $(i-1)\ast(j-1)$ $(1\leq i, j\leq4)$. Choose a positive quandle 2-cocycle
\begin{center}
$\phi=\chi_{(0, 1)}+\chi_{(1, 0)}+\chi_{(2, 0)}+\chi_{(0, 2)}+\chi_{(1, 2)}+\chi_{(2, 1)}\in H_{Q+}^2(S_4; \mathbf{Z}_2)$,
\end{center}
it was proved in \cite{Car2003} that $\Phi_{\phi}(3_1)=\Phi_{\phi}(4_1)=\sum\limits_40+\sum\limits_{12}1$.

We end this subsection by some remarks on the positive quandle 2-cocycle invariants with trivial quandles. First note that for $T_n$ and for any knot diagram there exist exactly $n$ trivial proper colorings. By the definition of $\pm$ quandle homology groups we can not obtain any new information from the $\pm$ quandle cocycle invariants. However it was pointed out in \cite{Car2003} that for any $\phi\in H_Q^2(T_n)$ and any link $L$, the quandle 2-cocycle invariant $\Phi_\phi(L)$ is a function of pairwise linking numbers. For example $\phi=\chi_{(a_1, a_2)}\in H_Q^2(T_2)$ can be used to distinguish the Hopf link from the trivial link. Since $H_{Q+}^2(T_2)\cong \mathbf{Z}_2$ with generator $\phi=\chi_{(a_1, a_2)}-\chi_{(a_2, a_1)}$, one obtains $\Phi_\phi(L)$ is trivial for any link $L$. In order to obtain some information from the link, we can work with coefficient $\mathbf{Z}_2$. In this way we can obtain the parity information of the pairwise linking numbers. For example, a link $L=K_1\cup \cdots \cup K_m$ is a proper link, i.e. $\sum\limits_{j\neq i}lk(K_i, K_j)=0$ $($mod 2$)$ for any $1\leq i\leq m$, if and only if $\sum\limits_{\rho_{1, m-1}}\prod\limits_{\tau}W_{\phi=\chi_{(a_1, a_2)}}(\tau, \rho_{1, m-1})=\sum\limits_m0$. Here $\mathbf{Z}_2=\{0, 1\}$ and $\rho_{1, m-1}$ denotes the set of proper colorings which assign one component with $a_1$ and the else with $a_2$. This result mainly follows from the fact that $H_{Q+}^2(X; \mathbf{Z}_2)\cong H_Q^2(X; \mathbf{Z}_2)$. From this viewpoint, for $T_n$, it seems that the positive quandle 2-cocycle invariant is a sort of $\mathbf{Z}_2$-version of the quandle 2-cocycle invariant. Later in the final section we will show that this is not the case.

\subsection{Positive quandle cocycle invariants for knotted surfaces}
In this subsection, with a given positive quandle 3-cocycle we will define a state-sum invariant for knotted surfaces in $R^4$. First we will take a short review of the background of knotted surfaces in $R^4$. The readers are referred to \cite{Car1998} and \cite{Car2004} for more details.

By a \emph{knotted surface} we mean an embedding $f$ of a closed oriented surface $F$ into $R^4$. Sometimes we also call the image $f(F)$ a knotted surface and denote it by $F$ for convenience. In particular when $F=S^2$ we name it a \emph{2-knot}. Two knotted surfaces are \emph{equivalent} if there exists an orientation preserving automorphism of $R^4$ which takes one knotted surface to the other. Similar to the knot diagram in knot theory, we usually study knotted surfaces via the knotted surface diagrams. Let $p: R^4\rightarrow R^3$ be the orthogonal projection from $R^4$ onto $R^3$, we may deform $f(F)$ slightly such that $p\circ f(F)$ is in a general position, then $p\circ f(F)$ is called a \emph{knotted surface diagram}. We must notice that a knotted surface diagram does not just mean an immersed surface in $R^3$. First there exist double points, triple points and branch points in $p\circ f(F)$. However it is well-known that $f(F)$ can be isotoped into a new position such that the projection contains no branch points \cite{Car1992,Gil1982}. Second, a knot diagram can be regarded as a 4-valent planar graph with some over-under information on each vertex. Hence a knotted surface diagram also contains the information of the over-sheet and under-sheet along the double curves. In other words, a knotted surface diagram is obtained from the projection by removing small open neighborhoods of the under-sheets along double curves.

Similar to the definition of the knot invariant Col$_X(K)$, we can define an integer-valued knotted surface invariant with a given quandle $X$. The main idea is using the elements of $X$ to color the regions of the broken surface diagram according to some rules at double curves. See the figure below, here $\overrightarrow{n}$ denotes the normal vector of the knotted surface diagram.
\begin{center}
\includegraphics{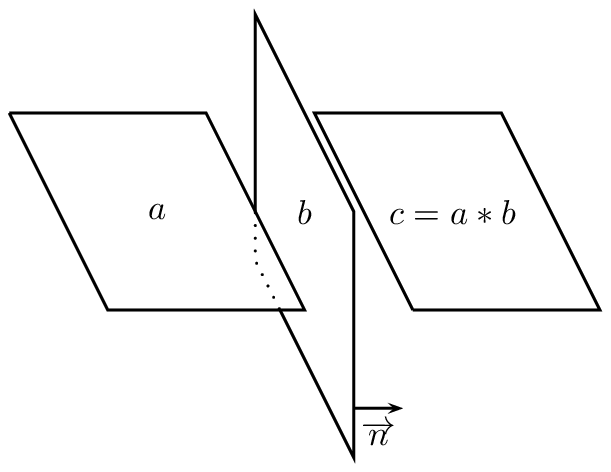}
\centerline{\small Figure 5: Coloring rules at a double curve\quad}
\end{center}

It is not difficult to check that the rule above is well-defined at each triple point \cite{Car2003}. Recall that different knotted surface diagrams represent the same knotted surface if and only if one of them can be achieved from the other by a finite sequence of Roseman moves \cite{Ros1998}. Similar to the proper coloring of knot diagrams, the number of the coloring satisfying the condition above is invariant under the Roseman moves, hence is a knotted surface invariant. We use Col$_X(F)$ to denote it.

The main idea of defining a knotted surface invariant with a positive quandle 3-cocycle is analogous to the definition of the quandle 3-cocycle invariant proposed in \cite{Car2003}. As a generalization of the counting invariant Col$_X(F)$, we need to assign an invariant for each colored knotted surface diagram and then take the sum of them. The position of triple point in knotted surface diagram is analogous to that of crossing point in knot diagram. Therefore this invariant can be obtained by assigning a weight to each triple point of the colored diagram.

Let $F$ be a knotted surface diagram and $X$ a finite quandle. Assume $G$ is an abelian group and $\theta\in H_{Q+}^3(X; G)$ is a positive quandle 3-cocycle. Consider the shadow of the diagram $F$, which is the immersed surface in $R^3$ without removing neighborhood along double curves. The shadow separates $R^3$ into several regions. It is not difficult to observe that we can use white and black to color these regions in 3-dimensional checkerboard fashion, i.e. adjacent regions are colored with different colors. We remark that the assumption that the surface is orientable is essentially used here. As before we assume that the unique unbounded region is colored white. For each triple point $\tau$ we can associate a sign $\epsilon(\tau)$ according to the figure below $($W=white, B=Black$)$.
\begin{center}
\includegraphics{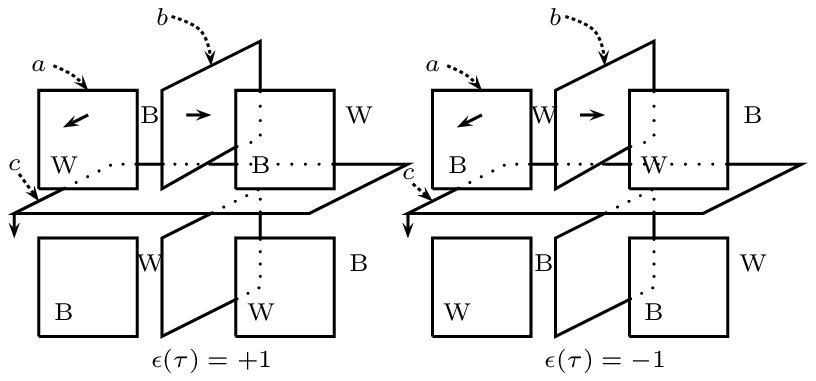}
\centerline{\small Figure 6: Signs of triple points\quad}
\end{center}

Let $\rho$ denote a coloring of $F$ by $X$. Assume $\tau$ is a triple point of $F$, the bottom, middle, top sheets around the octant from which all normal vectors point outwards are colored by $a, b, c$ respectively, see the figure above. Note that the sign of the triple point used here does not depend on the orientation of the surface. We associate a weight at the triple point $\tau$ as
\begin{center}
$W_{\theta}(\tau, \rho)=\theta(a, b, c)^{\epsilon(\tau)}\in G$.
\end{center}
Now we can define the \emph{positive quandle 3-cocycle invariant} of knotted surface $F$ associated with $\theta$ to be
\begin{center}
$\Theta_{\theta}(F)=\sum\limits_{\rho}\prod\limits_{\tau}W_{\theta}(\tau, \rho)\in \mathbf{Z}G$,
\end{center}
where $\rho$ runs all colorings of $F$ by $X$ and $\tau$ runs all triple points of the diagram.

We remark that the sign of a triple point has another definition. Consider the normal vectors of the top, middle and bottom sheets, if the orientation in this order matches the orientation of $R^3$, we say this triple point is positive. Otherwise it is negative. Replace $\epsilon(\tau)$ with the sign of triple point defined in this way one obtains the state-sum invariants introduced in \cite{Car2003}.
\begin{theorem}
The positive quandle 3-cocycle invariant $\Theta_{\theta}(F)$ is preserved under Roseman moves. If a pair of positive quandle 3-cocycles $\theta_1$ and $\theta_2$ are cohomologous, then $\Theta_{\theta_1}(F)=\Theta_{\theta_2}(F)$. In particular if $\theta$ is a coboundary, we have $\Theta_{\theta}(F)=\sum\limits_{Col_X(F)}1$.
\end{theorem}
\begin{proof}
We summarize the proof. There are only three types of Roseman move that involve triple points, see \cite{Car2003}. The first one creates or cancels a pair of triple points with oppositive signs, the second one moves a branch point through a sheet. The contribution of the two triple points in the first case will cancel out, and the contribution of the triple point in the second case is trivial according to the definition of positive quandle cohomology groups. Thus it suffices to prove that $\Theta_{\theta}(F)$ is invariant under tetrahedral move. See the figures below.
\begin{center}
\includegraphics{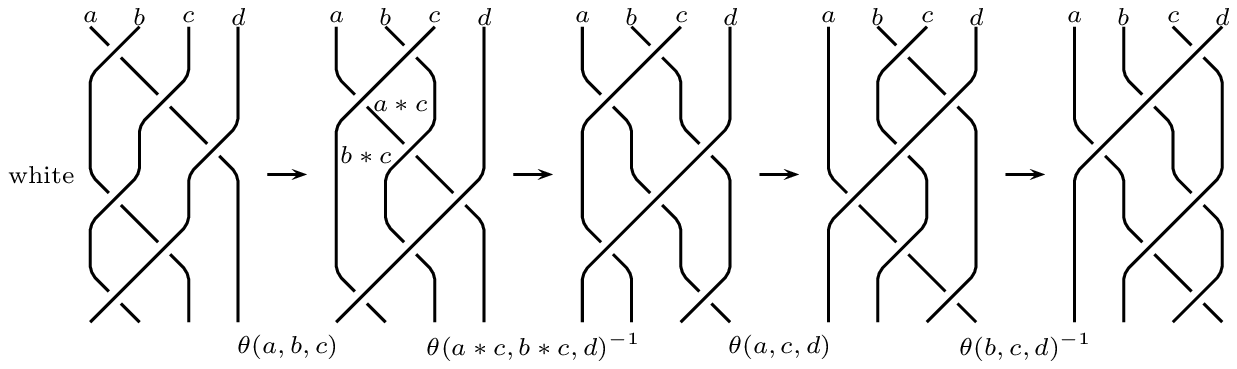}
\centerline{\small Figure 7: Left hand side of tetrahedral move\quad}
\end{center}
\begin{center}
\includegraphics{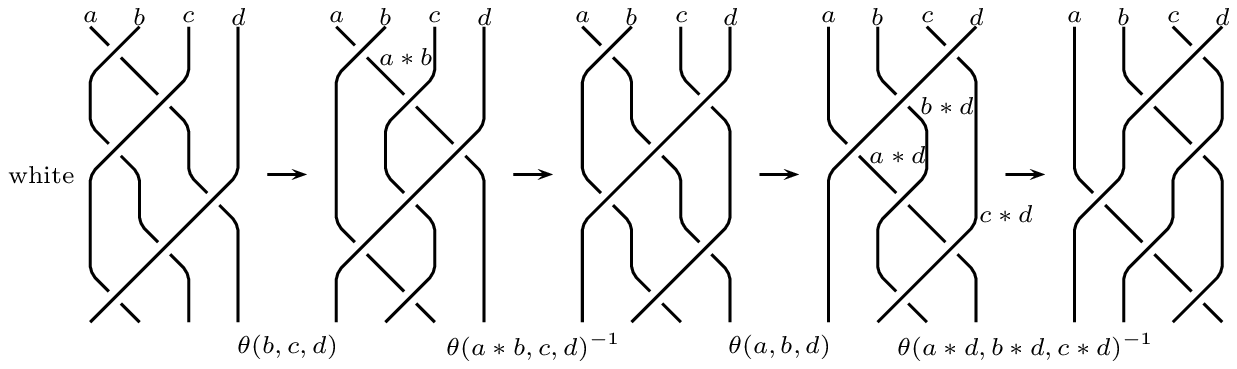}
\centerline{\small Figure 8: Right hand side of tetrahedral move\quad}
\end{center}

Here we use the movie description of knotted surface, see \cite{Car1998} for more detail. For example figure 7 contains five slices of a knotted surface according to a fixed height function, each slice consists of four sheets which are cross sections of four planes, and a pair of adjacent slices depict a triple point. Figure 7 and figure 8 correspond to the left hand side and the right hand side of tetrahedral move. Without loss of generality, suppose the leftmost region of each slice has the white color, and other regions can be colored in checkerboard fashion. The left hand side of tetrahedral move contributes $\theta(a, b, c)\theta(a\ast c, b\ast c, d)^{-1}\theta(a, c, d)\theta(b, c, d)^{-1}$ to $\Theta_{\theta}(F)$, and the right side has the contribution $\theta(b, c, d)\theta(a\ast b, c, d)^{-1}\theta(a, b, d)\theta(a\ast d, b\ast d, c\ast d)^{-1}$. In order to prove that $\Theta_{\theta}(F)$ is invariant under tetrahedral move, it suffices to show that
\begin{center}
$\theta(a, b, c)\theta(a\ast c, b\ast c, d)^{-1}\theta(a, c, d)\theta(b, c, d)^{-1}\theta(b, c, d)^{-1}\theta(a\ast b, c, d)\theta(a, b, d)^{-1}\theta(a\ast d, b\ast d, c\ast d)=1$
\end{center}
Comparing the equation above with the positive quandle 3-cocycle condition (note that the equation is written in multiplicative notation at present), we find that the condition $\theta\in H_{Q+}^3(X; G)$ guarantees the invariance of $\Theta_{\theta}(F)$. Here we only list one case of tetrahedral move, for other possible tetrahedral moves the invariance of $\Theta_{\theta}(F)$ can be proved in the same way.

Next we show that $\Theta_{\theta}(F)=\sum\limits_{Col_X(F)}1$ if $\theta$ is a coboundary. As we mentioned before, we can choose a knotted surface diagram such that the shadow of it contains no branch points. The double point set of it is a 6-valent graph and each vertex corresponds to a triple point. Fix a coloring $\rho$. According to the assumption that $\theta$ is a coboundary, i.e. $\theta=\delta_+^2\phi$ for some $\phi\in H_{Q+}^2(X; G)$, we have
\begin{center}
$\theta(a, b, c)=\delta_+^2\phi(a, b, c)=\phi(\partial_3^+(a, b, c))=\phi(b, c)^{-2}\phi(a, c)\phi(a\ast b, c)\phi(a, b)^{-1}\phi(a\ast c, b\ast c)^{-1}\in G$
\end{center}

Consider the triple point $\tau$ on the left side of figure 6, which has a weight $W_\theta(\tau, \rho)=\theta(a, b, c)=\phi(b, c)^{-2}\phi(a, c)\phi(a\ast b, c)\phi(a, b)^{-1}\phi(a\ast c, b\ast c)^{-1}$. There are six edges adjacent to the triple point $\tau$, two of them come from the intersection of the top sheet and the middle sheet, two of them come from the intersection of the middle sheet and the bottom sheet and the rest come from the intersection of the bottom sheet and the top sheet. We use $tm_1(\tau), tm_2(\tau), mb_1(\tau), mb_2(\tau), bt_1(\tau), bt_2(\tau)$ to denote these edges, where $tm_i(\tau)$ $(i=1, 2)$ denote the two edges belonging to the intersection of the top sheet and the middle sheet, $mb_i(\tau)$ $(i=1, 2)$ denote the two edges belonging to the intersection of the middle sheet and the bottom sheet and $bt_i(\tau)$ $(i=1, 2)$ denote the two edges belonging to the intersection of the bottom sheet and the top sheet. The order of the two edges belonging to the intersection of two sheets matches the orientation of the normal vector of the third sheet. Then the contribution of $\tau$ to $\Theta_{\theta}(F)$ can be separated into six parts: $\phi(b, c)^{-1}, \phi(b, c)^{-1}, \phi(a, b)^{-1}, \phi(a\ast c, b\ast c)^{-1}, \phi(a, c), \phi(a\ast b, c)$. We assign these six parts to $tm_1(\tau), tm_2(\tau), mb_1(\tau), mb_2(\tau), bt_1(\tau), bt_2(\tau)$ respectively. Therefore the contribution of $\tau$ can be regarded as the product of the contribution of the six edges adjacent to $\tau$. We remark that the contribution of each edge can be read directly from figure 5, the double line in figure 5 has contribution $\phi(a, b)^{\pm1}$. Here the sign of $\pm1$ is decided by the position of the two sheets. The sign is positive if the two sheets are the top sheet and the bottom sheet, for other cases the sign is negative. If sign of the triple point is negative then all the contribution will take the inverse.

In order to show that $\Theta_{\theta}(F)$ is trivial, it is sufficient to prove that each edge obtains opposite contributions from the two endpoints of it. We continue our discussion in two cases: two endpoints has the same sign or different signs.
\begin{center}
\includegraphics{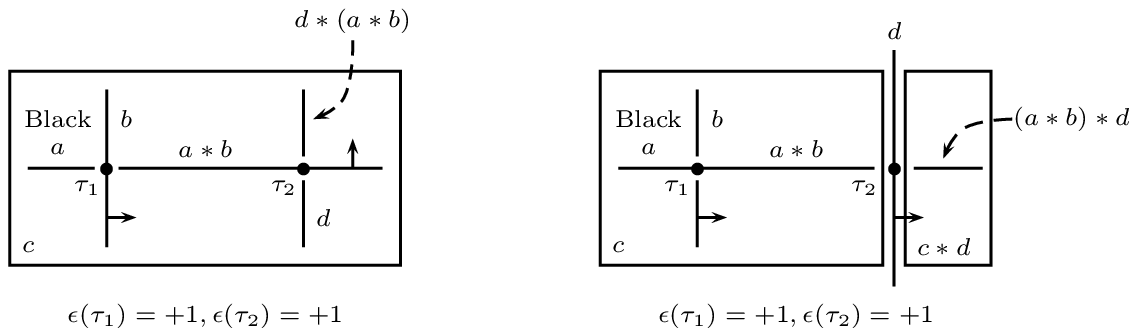}
\centerline{\small Figure 9: Two possibilities of adjacent triple points with the same sign\quad}
\end{center}
\begin{itemize}
  \item $\epsilon(\tau_1)=+1$ and $\epsilon(\tau_2)=+1$, there are two possibilities in this case. First consider the left side of figure 9. There are two triple points $\tau_1$ and $\tau_2$ with the same sign. Without loss of generality we assume the sign is positive. The frame with color $c$ denotes the top sheet of $\tau_1$ and $\tau_2$, and the straight lines are cross sections between the middle sheet or bottom sheet with the top sheet. Since $W_{\theta}(\tau_1, \rho)=\theta(a, b, c)$ and $W_{\theta}(\tau_2, \rho)=\theta(d, a\ast b, c)$, the contribution from $\tau_1$ to the edge with color $a\ast b$ is $\phi(a\ast b, c)$ and that from $\tau_2$ is $\phi(a\ast b, c)^{-1}$. The negative sign comes from the fact that for triple point $\tau_2$, the edge with color $a\ast b$ belongs to the intersection of the top sheet and the middle sheet. Hence the contributions from $\tau_1$ and $\tau_2$ to the edge between them cancel out. Consider the 6-valent graph consists of the double point set, it follows that the product of the contribution from each vertex to $\Theta_{\theta}(F)$ vanishes.

      For the right side of figure 9, we still have $\epsilon(\tau_1)=+1$ and $\epsilon(\tau_2)=+1$. Note that in this case the sheet with color $d$ is the top sheet of the triple point $\tau_2$. We have $W_{\theta}(\tau_1, \rho)=\theta(a, b, c)$ and $W_{\theta}(\tau_2, \rho)=\theta(a\ast b, c, d)$. Therefore the contribution from $\tau_1$ to the edge with color $a\ast b$ is $\phi(a\ast b, c)$ and the contribution from $\tau_2$ to the edge with color $a\ast b$ is $\phi(a\ast b, c)^{-1}$, since the edge with color $a\ast b$ belongs to the intersection of the middle sheet and the bottom sheet of $\tau_2$. Therefore the contributions from $\tau_1$ and $\tau_2$ to the edge between them still cancel out.
\end{itemize}
\begin{center}
\includegraphics{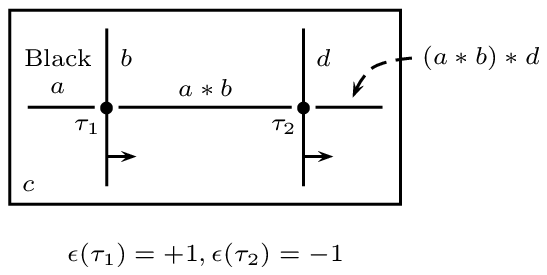}
\centerline{\small Figure 10: Adjacent triple points with different signs\quad}
\end{center}
\begin{itemize}
  \item $\epsilon(\tau_1)=+1$ and $\epsilon(\tau_2)=-1$, see figure 10. We can read from the figure that $W_{\theta}(\tau_1, \rho)=\theta(a, b, c)$ and $W_{\theta}(\tau_2, \rho)=\theta(a\ast b, d, c)^{-1}$. As before the contribution from $\tau_1$ to the edge with color $a\ast b$ is $\phi(a\ast b, c)$. Meanwhile, due to $\epsilon(\tau_2)=-1$, the contribution from $\tau_2$ to the edge with color $a\ast b$ equals $\phi(a\ast b, c)^{-1}$. Hence in this case we still have $\prod\limits_{\tau}W_{\theta}(\tau, \rho)=1$. The proof is finished.
\end{itemize}
\end{proof}

\textbf{Remark} In quandle cohomology theory, quandle 3-cocycle $\theta$ also can be used to define a state-sum invariant for knots via the shadow coloring. Given a knot diagram $K$ and a quandle $X$, a \emph{shadow coloring} of $K$ by $X$ is a function from the set of arcs of $K$ and the regions separated by the shadow of $K$ to the quandle $X$, satisfying the coloring condition depicted below.
\begin{center}
\includegraphics{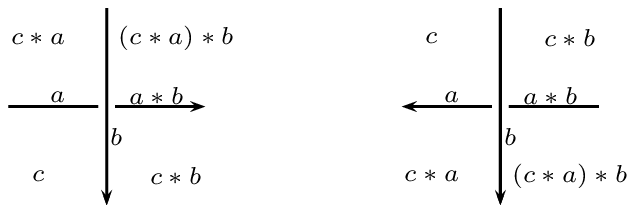}
\centerline{\small Figure 11: Shadow coloring at a crossing\quad}
\end{center}
It is not difficult to observe that shadow colorings are completely decided by the proper colorings on arcs and the color of one fixed region. Hence the number of shadow colorings do not offer any new information rather than Col$_X(K)$. Given a quandle 3-cocycle $\theta\in H_Q^3(X; G)$ one can associate a weight $W_{\theta}(\tau, \widetilde{\rho})=\theta(c, a, b)^{w(\tau)}$ with the crossing point in figure 11, here $w(\tau)$ means the writhe of the crossing and $\widetilde{\rho}$ denotes a shadow coloring. Then the element of $\mathbf{Z}G$ : $\Psi_{\theta}(K)=\sum\limits_{\widetilde{\rho}}\prod\limits_{\tau}W_{\theta}(\tau, \widetilde{\rho})$ defines a knot invariant, where $\widetilde{\rho}$ runs all shadow colorings and $\tau$ runs all crossing points. It was pointed out in \cite{Rou2000} that this state-sum invariant can be used to detect the chirality of the trefoil knot. An interesting question is how to define a knot invariant with a given positive quandle 3-cocycle.

\section{On trivially colored crossing points}
We end this paper with two elementary examples which concerns trivially colored crossing points. Given a knot diagram $K$ and a quandle $X$, choose a crossing point $\tau$ of the knot diagram. We say $\tau$ is a \emph{trivially colored crossing point} if for any proper coloring of $K$ by $X$, the over-arc and the two under-arcs of $\tau$ are labelled with the same color. For example the crossing point involved in the first Reidemeister move is a trivially colored crossing point for any given quandle. As another instance, consider the crossing $\tau$ of the knot diagram below. If we take $X=R_3$, then the crossing $\tau$ is a trivially colored crossing point.
\begin{center}
\includegraphics{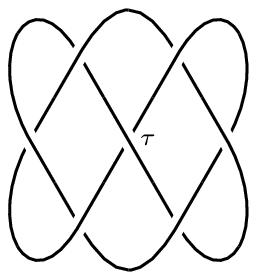}
\centerline{\small Figure 12: A trivially crossing point\quad}
\end{center}

There are two reasons for us to study trivially colored crossing points. The first motivation comes from the Kauffman-Harary conjecture. L. Kauffman and F. Harary \cite{Har1999} conjectured that the minimum number of distinct colors that are needed to produce a non-trivial Fox $n$-coloring of a reduced alternating knot diagram $K$ with prime determinate $n$ equals the crossing number of $K$. In other words for any non-trivial Fox $n$-coloring of $K$, different arcs are assigned by different colors. In 2009 this conjecture was settled by T.W. Mattman and P. Solis in \cite{Mat2009}. It means that for a given reduced alternating diagram with prime determinate $n$ and the quandle $R_n$, no crossing point of the knot diagram is trivially colored. However this conjecture does not hold if we ignore the condition of prime determinate. For example consider the standard diagram of the connected sum of two reduced alternating knot diagrams which have prime determinate $m$ and $n$ respectively. Choose the quandle $R_{mn}$. Now there exists no Fox $mn$-coloring such that different arcs has different colors, but for each crossing point there exists a proper coloring such that this cross point is nontrivially colored. It is possible to extend the range of knots in Kauffman-Harary conjecture by replacing the heterogeneity of the coloring with the nonexistence of trivially colored crossing points.

The second motivation of investigating trivially colored crossing points arises from the $\pm$ quandle 2-cocycle invariants. Recall the definition of $\pm$ quandle cohomology groups, in order for the 2-cocycle invariant to be preserved under the first Reidemeister move we put $\phi(a, a)=1$. In this way the first Reidemeister move has no effect on the 2-cocycle invariant, but the disadvantage is the information of trivially colored crossing points are also lost. For instance if a crossing point $\tau$ of a knot diagram $K$ is a trivially colored crossing point $($associated with $X)$, then $W_{\phi}(\tau, \rho)=1$ for any 2-cocycle $\phi$ and proper coloring $\rho$. Hence it has no contribution to the cocycle invariant.

The first example we want to discuss is the Borromean link. The Borromean link is a nontrivial 3-component link with trivial proper sublinks. The Borromean link is nontrivial follows from the fact that one component of the Borromean represents a commutator of the fundamental group of the complement of the other two components \cite{Rol1976}. Let $X=T_n$, as we mentioned before, the quandle 2-cocycle of a link is a function of pairwise linking numbers \cite{Car2003}. Since the pairwise linking numbers of the Borromean link are all trivial, it follows that the quandle 2-cocycle invariant can not distinguish the Borromean link from the trivial link. However we can use a refinement of the positive quandle 2-cocycle invariant to show that the Borromean link is nontrivial.
\begin{center}
\includegraphics{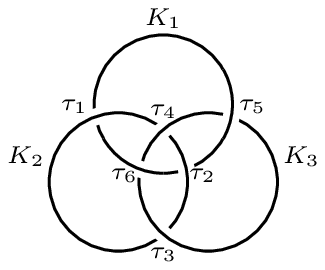}
\centerline{\small Figure 13: The Borromean link\quad}
\end{center}

Let $K_1, K_2, K_3$ denote the three components of the Borromean link and $\tau_i$ $(1\leq i\leq6)$ denote the crossing points of it. See the figure above. According to the definition of $\epsilon(\tau_i)$ we used in Section 5, we have $\epsilon(\tau_i)=+1$ $(1\leq i\leq6)$. Take $\phi=\chi_{(a_1, a_2)}+\chi_{(a_2, a_1)}\in H_{Q+}^2(T_2; \mathbf{Z}_4)$, consider the element
\begin{center}
$\widetilde{\Phi}_{\phi}(BL)=\sum\limits_{\rho}(t_1^{W_{\phi}(\tau_1, \rho)+W_{\phi}(\tau_2, \rho)}t_2^{W_{\phi}(\tau_3, \rho)+W_{\phi}(\tau_4, \rho)}t_3^{W_{\phi}(\tau_5, \rho)+W_{\phi}(\tau_6, \rho)})\in \mathbf{Z}[t_1, t_2, t_3]/(t_1^4=t_2^4=t_3^4=1)$,
\end{center}
where $W_{\phi}(\tau_i, \rho)$ is the weight associated to the crossing $\tau_i$ and $\rho$ runs all proper colorings of the diagram in figure 13 by $T_2$. In general for a diagram of a 3-component link $L=K_1\cup K_2\cup K_3$, we define
\begin{center}
$\widetilde{\Phi}_{\phi}(L)=\sum\limits_{\rho}(t_1^{\sum\limits_{\tau\in K_1\cap K_2}W_{\phi}(\tau, \rho)}t_2^{\sum\limits_{\tau\in K_2\cap K_3}W_{\phi}(\tau, \rho)}t_3^{\sum\limits_{\tau\in K_3\cap K_1}W_{\phi}(\tau, \rho)})\in \mathbf{Z}[t_1, t_2, t_3]/(t_1^4=t_2^4=t_3^4=1)$,
\end{center}
where $K_i\cap K_j$ denotes the set of crossing points between $K_i$ and $K_j$ and $\rho$ runs all proper colorings of the diagram by $T_2$.
\begin{proposition}
$\widetilde{\Phi}_{\phi}(L)$ is invariant under Reidemeister moves.
\end{proposition}
\begin{proof}
The result mainly follows from the fact that $\phi=\chi_{(a_1, a_2)}+\chi_{(a_2, a_1)}\in H_{Q+}^2(T_2; \mathbf{Z}_4)$.
\end{proof}

Direct calculation shows that $\widetilde{\Phi}_{\phi}(BL)=2+2t_1^2t_2^2+2t_2^2t_3^2+2t_3^2t_1^2$ and $\widetilde{\Phi}_{\phi}(TL)=8$, where $BL$ denotes the Borromean link and $TL$ denotes the 3-component trivial link. Therefore $\widetilde{\Phi}_{\phi}(L)$ can be used to distinguish the Borromean link from the trivial link. Further, since we are working with $T_2$, it follows that $\widetilde{\Phi}_{\phi}(L)$ is invariant under self-crossing changes. Hence the result above shows that the Borromean link is not link-homotopic to the 3-component trivial link. Essentially speaking, the reason why $\widetilde{\Phi}_{\phi}(L)$ can tell the difference between the Borromean link and the trivial link is that the Borromean link is alternating. The writhe of a crossing between two components does not depend on the position of the third component, hence if the linking number of two components is zero then the third component has no effect on the quandle 2-cocycle invariant $($associated with $T_n)$. However the sign $\epsilon(\tau)$ we used here contains some information of the position of the third component. This is the reason why positive quandle 2-cocycle can be used to distinguish the Borromean link and the trivial link. We remark that although for any quandle 2-cocycle of $T_n$ the state-sum invariant can not distinguish the Borromean link and the trivial link, in \cite{Ino2012} A. Inoue used a 2-cocycle of a quasi-trivial quandle to show that the Borromean link is not link-homotopic to the 3-component trivial link. Note that the link-homotopy invariants defined by A. Inoue in \cite{Ino2012} have the same value on the Borromean link and the 3-component trivial link if we work with the trivial quandles.

The second example concerns the Fox 3-coloring. As we mentioned before, the diagram of knot $7_4$ in figure 12 contains a trivially colored crossing point if we consider the Fox 3-colorings. A natural question is which kind of knot diagram contains a trivially colored crossing point $($associated with $R_3)$. For example if the determinate of the knot is not divisible by 3 then there exists no nontrivial Fox 3-coloring, hence each crossing point is a trivially colored crossing point. We end this paper by a simple sufficient condition to this question, which shows that the knot diagram in figure 12 contains a trivially colored crossing point without needing to list all the proper colorings.
\begin{proposition}
Let $K$ be a knot diagram, consider the Fox 3-colorings, if $\sum\limits_{\tau}\epsilon(\tau)$ is not divisible by 3, then $K$ contains at least one trivially colored crossing point.
\end{proposition}
\begin{proof}
Recall that $R_3=\{0, 1, 2\}$ with quandle operations $i\ast j=2j-i$ $($mod 3$)$. Consider the coboundary
\begin{center}
$\phi=\chi_{(0, 1)}+\chi_{(1, 0)}+\chi_{(1, 2)}+\chi_{(2, 1)}+\chi_{(2, 0)}+\chi_{(0, 2)}\in H_{Q+}^2(R_3; \mathbf{Z}_3)$.
\end{center}
Since $\phi=\delta\chi_0$ it follows that $\Phi_{\phi}(K)=\sum\limits_{Col_3(K)}0$ $($here we write $\mathbf{Z}_3=\{0, 1, 2\})$. On the other hand, for each nontrivially colored crossing point $\tau$, the contribution of $\tau$ to $\Phi_{\phi}(K)$ is $\epsilon(\tau)$. Therefore if $K$ contains no trivially colored crossing points we have $\sum\limits_{\tau}\epsilon(\tau)=0$ $($mod 3$)$. The result follows.
\end{proof}


\begin{thebibliography}{99}
\bibitem{Car1992}J.S. Carter, M. Saito. Canceling branch points on the projections of surfaces in 4-space. Proc. Amer. Math. Soc., 116, 1992, 229-237
\bibitem{Car1998}J.S. Carter, M. Saito. Knotted surfaces and their diagrams. The American Mathematical Society, 1998
\bibitem{Car2001J}J.S. Carter, D. Jelsovsky, S. Kamada, M. Saito. Quandle homology groups, their Betti numbers, and virtual knots. J. Pure. Appl. Algebra, 157, 2001, 135-155
\bibitem{Car2001A}J.S. Carter, D. Jelsovsky, S. Kamada, M. Saito. Computations of quandle cocycle invariants of knotted curves and surfaces. Adv. in Math. 157, 2001, 36-94
\bibitem{Car2003}J.S. Carter, D. Jelsovsky, S. Kamada, L. Langford, M. Saito. Quandle cohomology and state-sum invariants of knotted curves and surfaces. Trans. Amer. Math. Soc., 355, 2003, 3947-3989
\bibitem{Car2004}J.S. Carter, S. Kamada, M. Saito. Sufaces in 4-space. Encyclopaedia of Mathematical Sciences, 142, Springer Verlag, 2004
\bibitem{Car2012}J.S. Carter. A survey of quandle ideas. Introductory lectures on knot theory, 22-53, Series on Knots and Everything, 46, World Scientific Publishing, Hackensack, NJ 2012
\bibitem{Cla2013}W.E. Clark, M. Elhamdadi, M. Saito, T. Yeatman. Quandle colorings of knots and applications. arXiv:1312.3307v1
\bibitem{Eis1999}M. Eisermann. The number of knot group representations is not a Vassiliev invariant. Proc. Amer. Math. Soc., 128, 1999, 1555-1561
\bibitem{Eis2003}M. Eisermann. Homological characterization of the unknot. J. Pure Appl. Algebra, 177, 2003, 131-157
\bibitem{Fen1992}R. Fenn, C. Rourke. Racks and links in codimension two. Journal of Knot Theory and Its Ramificaitons, 1, 1992, 343-406
\bibitem{Fen1995}R. Fenn, C. Rourke, B. Sanderson. Trunks and classifying spaces. Appl. Categ. Struct., 3, 1995, 321-356
\bibitem{Fen1996}R. Fenn, C. Rourke, B. Sanderson. James bundles and applications. Preprint, 1996
\bibitem{Fox1961}R. H. Fox. A quick trip through knot theory. M. K. Fort Jr.(Ed.), Topology of 3-manifolds and related topics, GA, 1961, 120-167
\bibitem{Gil1982}C. Giller. Towards a classical knot theory for surfaces in $R^4$. Illinois Journal of Mathematics, 26, 1982, 591-631
\bibitem{Har1999}F. Harary, L.H. Kauffman. Knots and graphs I-Arc graphs and colorings. Adv. in Appl. Math., 22, 1999, 312-337
\bibitem{Ho2005}B. Ho, S. Nelson. Matrices and finite quandles. Homology, Homotopy and Applications, 7, 2005, 197-208
\bibitem{Ino2012}A. Inoue. Quasi-triviality of quandles for link-homotopy. arXiv:math.GT/1205.5891
\bibitem{Ino2013}A. Inoue, Y. Kabaya. Quandle homology and complex volume. Geom Dedicata, August 2013
\bibitem{Joy1982}D. Joyce. A classifying invariant of knots, the knot quandle. J. Pure Appl. Algebra, 23, 1982, 37-65
\bibitem{Kam2002}S. Kamada. Knot invariants derived from quandles and racks. Geometry \& Topology Monographs, 4, 2002, 103-117
\bibitem{Lit2003}L.N. Litherland, S. Nelson. The Betti number of some finite racks. J. Pure Appl. Algebra, 178, 2003, 187-202
\bibitem{Mat2009}T.W. Mattman, P. Solis. A proof of Kauffman-Harary conjecture. Algebraic \& Geometric Topology, 9, 2009, 2027-2039
\bibitem{Mat1984}S. V. Matveev. Distributive groupoids in knot theory. Math. USSR Sb. 47, 1984, 73-83
\bibitem{Moc2003}T. Mochizuki. Some calculations of cohomology groups of finite Alexander quandles. J. Pure Appl. Algebra, 179, 2003, 287-330
\bibitem{Nie2009}M. Niebrzydowski, J.H. Przytycki. Homology of dihedral quandles. J. Pure Appl. Algebra, 213, 2009, 742-755
\bibitem{Pol2010} Michael Polyak. Minimal generating sets of Reidemeister moves. Quantum Topology, 1, 2010, 399-411
\bibitem{Prz1998}J.H. Przytycki. 3-coloring and other elementary invariants of knots. Knot Theory, Bnanch Center Publications, 42, 1998
\bibitem{Rol1976}D. Rolfsen. Knots and links. Publish or Perish Press, 1976
\bibitem{Ros1998}D. Roseman. Reidemeister-type moves for surfaces in four dimensional space. Banach Center Publication 42 Knot Theory, 1998, 347-380
\bibitem{Rou2000}C. Rourke, B. Sanderson. There are two 2-twist-spun trefoils. arXiv:math.GT/0006062
\bibitem{Sai2010}Masahico Saito. The minimum number of Fox colors and quandle cocycle invariants. Journal of Knot Theory and Its Ramifications, 19, 2010, 1449-1456
\bibitem{Sat2004}S. Satoh, A. Shima. The 2-twist-spun trefoil has the triple point number four. Trans. Amer. Math. Soc., 356, 2004, 1007-1024
\bibitem{Tak1942}M. Takasaki. Abstraction of symmetric transformations. (in Japanese) Tohoku Math J., 49, 1942/43, 145-207
\bibitem{Ven2012}L. Vendramin. On the classification of quandles of low order. Journay of Knot Theory and Its Ramifications, 21, 2012, 1250088
\end{thebibliography}
\end{document}